\documentclass[11pt,letterpaper]{nyjm}
\date{July 31, 2016}

\usepackage{amssymb,amsmath}
\usepackage[dvips]{graphics}
\usepackage[dvips]{color}
\usepackage{graphicx}
\usepackage{hyperref}
\usepackage[english]{babel}
\usepackage[T1]{fontenc}       		
\usepackage{mathrsfs}   
\usepackage{comment}
\usepackage[textsize=tiny]{todonotes}
\usepackage{xcolor}

\newtheorem{Theorem}{Theorem}
\newtheorem{Lemma}[Theorem]{Lemma}
\newtheorem{Proposition}[Theorem]{Proposition}
\newtheorem{Corollary}[Theorem]{Corollary}
\newtheorem{Conjecture}[Theorem]{Conjecture}

\theoremstyle{definition}
\newtheorem{Definition}[Theorem]{Definition}
\newtheorem{Example}[Theorem]{Example}
\newtheorem{Remark}[Theorem]{Remark}
\newtheorem{Claim}{Claim}

\providecommand{\ch}[1]{\text{\raise 2pt \hbox{$\chi$}\kern-0.2pt}_{#1}}

\def\XXint#1#2#3{{\setbox0=\hbox{$#1{#2#3}{\int}$}
      \vcenter{\hbox{$#2#3$}}\kern-.5\wd0}}

\newcommand{\N}{\mathbb{N}}     
\newcommand{\R}{\mathbb{R}}     

\newcommand{\calB}{\mathscr{B}}

\newcommand{\calR}{\mathscr{R}}

\newcommand{\btheta}{\boldsymbol{\theta}}

\DeclareMathOperator{\diam}{diam}			

\usepackage[mathcal]{eucal}					

\def\sc{\sqcup}

\def\L{{{\mathcal{L}}}}

\renewcommand{\O}{\emptyset}

\renewcommand{\geq}{\geqslant}
\renewcommand{\leq}{\leqslant}
\renewcommand{\epsilon}{\varepsilon}

\newcommand{\liminfe}{\mathop{\underline{\lim}}}
\newcommand{\limsupe}{\mathop{\overline{\lim}}}

\usepackage{color}

\begin{document}
\author{Laurent Moonens}

\title[Differentiation in lacunary directions]{Differentiating along rectangles\\ in lacunary directions}

\begin{abstract}We show that, given some lacunary sequence of angles $\btheta=(\theta_j)_{j\in\N}$ not converging too fast to zero, it is possible to build a rare differentiation basis $\calB$ of rectangles parallel to the axes that differentiates $L^1(\R^2)$ while the basis $\calB_{\btheta}$ obtained from $\calB$ by allowing its elements to rotate around their lower left vertex by the angles $\theta_j$, $j\in\N$, fails to differentiate all Orlicz spaces lying between $L^1(\R^2)$ and $L\log L(\R^2)$.\end{abstract}


\thanks{This work was partially supported by the French ANR project ``GEOMETRYA'' no.~ANR-12-BS01-0014.}

\address{Laboratoire de Math\'ematiques d'Orsay, Universit\'e Paris-Sud, CNRS UMR8628, Universit\'e Paris-Saclay, B\^atiment 425, F-91405 Orsay Cedex, France.}

\email{Laurent.Moonens@math.u-psud.fr}
     
\maketitle

\section{Introduction}

Assume that $\btheta=(\theta_j)_{j\in\N}\subseteq (0,2\pi)$ is a lacunary sequence going to zero and denote by $\calB_{\btheta}$ the set of all rectangles in $\R^2$, one of whose sides makes an angle $\theta_j$ with the horizontal axis, for some $j\in\N$. It follows from results by {\scshape C\'ordoba} and {\scshape Fefferman} \cite{CF1977} (for $p>2$) and {\scshape Nagel, Stein} and {\scshape Wainger} \cite{NSW1978} (for all $p>1$) that for every $f\in L^p(\R^2)$, one has:
\begin{equation}\label{eq.db}
f(x)=\lim_{\begin{subarray}{c}R\in\calB_{\btheta}\\ R\ni x\\\diam R\to 0\end{subarray}} \frac{1}{|R|}\int_ R f,
\end{equation}
for almost every $x\in\R^2$ (we say, in this case, that $\calB_{\btheta}$ \emph{differentiates $L^p(\R^n)$}). This is often equivalent, according to Sawyer-Stein principles (see \emph{e.g.} {\scshape Garsia} \cite[Chapter 1]{GARSIA}), to the fact that the associated maximal operator $M_\calB$, defined for measurable functions $f$ by:
$$
M_\calB f(x):=\sup_{\begin{subarray}{c}R\in\calB\\R\ni x\end{subarray}} \frac{1}{|R|}\int_ R |f|,
$$
satisfies a weak $(p,p)$ inequality, \emph{i.e.} verifies:
$$
|\{M_\calB f>\alpha\}|\leq\frac{C}{\alpha^p}\int_{\R^2} |f|^p,
$$
for all $\alpha>0$ and all $f\in L^p(\R^2)$. By interpolation, of course, such a property for all $p>1$ implies that $M_\calB$ sends boundedly $L^p(\R^n)$ into $L^p(\R^n)$ for all $p>1$.

Since then, the $L^p$ ($p>1$) behaviour of the operators $M_{\calB_{\btheta}}$ has been studied when the lacunary sequence $\btheta$ is replaced by some Cantor sets (see \emph{e.g.} {\scshape Katz} \cite{KATZ1996} and {\scshape Hare} \cite{HARE2000}); recently, {\scshape Bateman} \cite{BATEMAN} obtained necessary and sufficient (geometrical) conditions on $\btheta$ providing the $L^p$ boundedness of $M_{\calB_{\btheta}}$.

In this paper we explore the behaviour of some maximal operators associated to rare differentiation bases of rectangles oriented in a lacunary set of directions $\btheta=\{\theta_j:j\in\N\}$, provided that the sequence $(\theta_j)$ does not converge too fast to zero. More precisely, we prove the following theorem.
\begin{Theorem}\label{thm.1}
Given a lacunary sequence $\btheta=(\theta_j)_{j\in\N}\subseteq (0,2\pi)$ satisfying:
$$
0<\liminfe_{j\to\infty} \frac{\theta_{j+1}}{\theta_j}\leq \limsupe_{j\to\infty} \frac{\theta_{j+1}}{\theta_j}<1,
$$
there exists a differentiation basis $\calB$ of rectangles parallel to the axes satisfying the two following properties:
\begin{enumerate}
\item[(i)] $M_\calB$ has weak type $(1,1)$ (in particular $\calB$ differentiates $L^1(\R^2)$);
\item[(ii)] if we denote by $\calB_{\btheta}$ the differentiation basis obtained from $\calB$ by allowing its elements to rotate around their lower left corner by any angle $\theta_j$, $j\in\N$, then for any Orlicz function $\Phi$ (see below for a definition) satisfying $\Phi=o(t\log_+ t)$ at $\infty$, the maximal operator$M_{\calB_{\btheta}}$ \emph{fails to have} weak type $(\Phi,\Phi)$ (in particular $\calB_{\btheta}$ fails to differentiate $L^\Phi(\R^n)$).
\end{enumerate}
\end{Theorem}
\begin{Remark}
The differentiation basis $\calB$ we shall construct in the proof of Theorem~\ref{thm.1} is \emph{rare}: it will be obtained as the smallest translation-invariant basis containing a \emph{countable} family of rectangles with lower left corner at the origin (see section~\ref{sec.3} for a more precise statement). 
\end{Remark}

Our paper is organized as follows: we first discuss some easy geometrical facts concerning rectangles and rotations along lacunary sequences, following with a proof of Theorem~\ref{thm.1}.

\section{Some basic geometrical facts}

In the sequel we always call \emph{standard rectangle} in $\R^2$ a set of the form $Q=[0,L]\times [0,\ell]$ where $L>0$ and $\ell>0$ are real numbers; we then let $Q_+:=[L/2,L]\times [0,\ell]$. For $\theta\in [0,2\pi)$ we also denote by $r_\theta$ the (counterclockwise) rotation of angle $\theta$ around the origin.

\begin{Lemma}\label{cl.1}
Fix real numbers $0\leq\vartheta<\theta<\frac{\pi}{2}$ and $0<2\ell<L$ and let $Q:=[0,L]\times [0,\ell]$. If moreover one has $ \tan(\theta-\vartheta)\geq 1/\sqrt{\frac 14 \left( \frac{L}{\ell}\right)^2-1}$, then $r_\vartheta Q_+$ and $r_\theta Q_+$ are disjoint.
\end{Lemma}
\begin{figure}[h]\begin{center}\includegraphics[width=10cm]{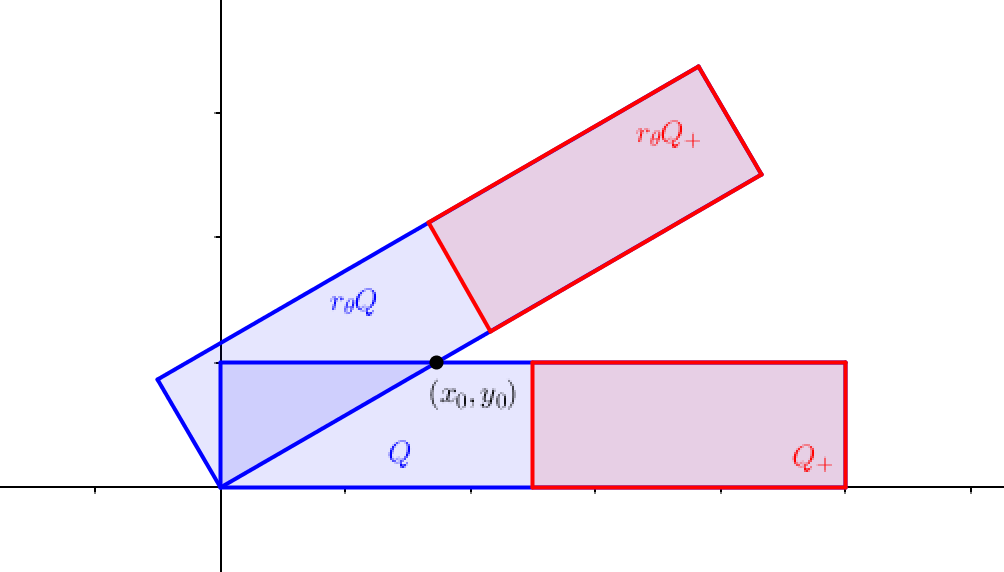} \caption{The rectangles $Q, Q_+, r_\theta Q$ and $r_\theta Q_+$}\label{fig.1}\end{center}\end{figure}
\begin{proof}
To prove this lemma, we can assume, without loss of generality, that one has $\vartheta=0$ (for otherwise, apply $r_{-\vartheta}$ to $r_\vartheta Q_+$ and $r_\theta Q_+$). Let $m:=\tan \theta$. Observe then that the lines $y=\ell$ and $y=m x$ intersect at $x_0=\ell/m\leq \ell \sqrt{\frac 14 \left( \frac{L}{\ell}\right)^2-1}\leq \frac L2$ and $y_0=\ell$. Since we also have:
$$
|(x_0,y_0)|\leq \ell\sqrt{\frac 14 \left(\frac{L}{\ell}\right)^2}=\frac L2,
$$
this shows indeed that $Q_+$ and $r_\theta Q_+$ are disjoint (see Figure~\ref{fig.1}).
\end{proof}

\begin{Lemma}\label{cl.2}
Assume that the sequence $(\theta_j)_{k\in\N}\subseteq (0,\pi/2)$ is such that one has:
\begin{equation}\label{eq.bilac}
0<\lambda<\liminfe_{j\to\infty} \frac{\theta_{j+1}}{\theta_j}\leq \limsupe_{j\to\infty} \frac{\theta_{j+1}}{\theta_j}<\mu<1.
\end{equation}
Let $\btheta:=\{\theta_j:j\in\N\}$.

There exists constants $d(\mu)>c(\mu)>0$ depending only on $\mu$ such that, for each $\epsilon>0$ and each integer $k\in\N^*$, one can find a standard rectangle $Q_k=[0,L_k]\times [0,\ell_k]$ and a subset $\btheta_k\subset\btheta$ satisfying $\#\btheta_k=k$ such that the following hold:
\begin{enumerate}
\item[(i)] $0\leq 2\ell_k\leq L_k\leq \epsilon$;
\item[(ii)] $c(\mu)\lambda^{-k}\leq\frac{L_k}{\ell_k}\leq d(\mu)\lambda^{-k}$;
\item[(iii)] $\left| \bigcup_{\theta\in\btheta_k} r_{\theta} Q_k\right|\geq \frac{k}{2}|Q_k|$.
\end{enumerate}
\end{Lemma}
\begin{proof}
To prove this lemma, observe first that letting $m_j:=\tan \theta_j$ for all $j\in\N$, one clearly has:
$$
\lim_{j\to\infty}\frac{m_j}{\theta_j}=1,
$$
so that (\ref{eq.bilac}) also holds for the sequence $(m_j)_{j\in\N}$. There hence exists an index $j_0\in\N$ such that, for all $j\geq j_0$, one has $\lambda\leq\frac{m_{j+1}}{m_j}\leq \mu$ (we may also and will assume that one has $m_{j_0}\leq 1$). For the sake of clarity, we shall now consider that $j_0=0$ and compute, for an integer $0\leq j<k$:
$$
\tan(\theta_j-\theta_k)=\frac{m_j-m_k}{1+m_jm_k}\geq\frac 12 (m_j-m_k).
$$
Since we also have, for every integer $0\leq j<k$:
$$
\lambda^{k-j} m_j\leq m_k\leq \mu^{k-j} m_j,
$$
we obtain under the same assumptions on $j$:
$$
\tan(\theta_j-\theta_k)\geq\frac 12 (m_j-m_k)\geq \frac 12 (\mu^{j-k}-1)m_k\geq \frac 12 \lambda^k (\mu^{-1}-1) m_0.
$$
Now choose real numbers $0\leq 2\ell\leq L\leq\epsilon$ (we write $L$ and $\ell$ instead of $L_k$ and $\ell_k$ here, for the index $k$ remains constant all through the proof) satisfying:
$$
\left(\frac{L}{\ell}\right)^2= 4+\lambda^{-2k} [(\mu^{-1}-1)m_0]^{-2}.
$$
It is clear that one has:
$$
\frac{L}{\ell}=\lambda^{-k} \sqrt{4 \lambda^{2k}+[(\mu^{-1}-1)m_0]^{-2}},
$$
so that (ii) holds if we take, for example, $c(\mu):=\sqrt{[(\mu^{-1}-1)m_0]^{-2}}$ and $d(\mu):=\sqrt{4+[(\mu^{-1}-1)m_0]^{-2}}$. On the other hand, (i) is clearly satisfied by assumption.

In order to show (iii), define $Q:=[0,L]\times [0,\ell]$ and observe that one has $$
\tan (\theta_j-\theta_k)\geq \frac{1}{\sqrt{\frac 14 \left(\frac{L}{\ell}\right)^2-1}},
$$
for all integers $j$ satisfying $j<k$. According to Lemma~\ref{cl.1}, this ensures that the family $\{r_{\theta_j}Q_+:j\in\N, j<k\}$ consists of pairwise disjoints sets; in particular we get:
$$
\left| \bigcup_{j=0}^{k-1} r_{\theta_j} Q\right|\geq \left| \bigsqcup_{j=0}^{k-1} r_{\theta_j} Q_+\right| =k\cdot \frac{|Q|}{2},
$$
(we used $\sqcup$ to indicate a disjoint union) and the lemma is proved.
\end{proof}

We now turn to studying maximal operators associated to families of standard rectangles.
\section{Maximal operators associated to lacunary sequences of directions}\label{sec.3}
From now on, given a family $\calR$ of standard rectangles and a set $\btheta\subseteq [0,2\pi)$, we let $r_{\btheta} \calR:=\{ r_{\theta} Q: Q\in\calR, \theta\in\btheta\}$, and we define, for $f:\R^2\to\R$ measurable:
$$
M_\calR f(x):=\sup\left\{\frac{1}{|Q|}\int_{\tau(Q)} |f|: Q\in\calR, \tau\text{ translation}, x\in\tau(Q)\right\},
$$
and:
$$
M_{r_{\btheta}\calR} f(x):=\sup\left\{\frac{1}{|R|}\int_{\tau(R)} |f|: R\in r_{\btheta}\calR, \tau\text{ translation}, x\in\tau (R)\right\}.
$$
Notice, in particular, that in case one has $\inf\{\diam R:R\in\calR\}=0$, $M_\calR$ and $M_{r_{\btheta}\calR}$ are the maximal operators associated to the translation-invariant differentiation bases $\calB$ and $\calB_{\btheta}$ defined respectively by: $$\calB:=\{\tau(Q):Q\in\calR, \tau\text{ translation}\}$$ and $$
\calB_{\btheta}:=\{\tau(r_\theta Q):Q\in\calR, \theta\in\btheta,\tau\text{ translation}\}.
$$

The next proposition will be useful in order to study the maximal operator $M_{r_{\btheta}\calR}$.
Observe that it has the flavour of {\scshape Stokolos}' \cite[Lemma~1]{STOKOLOS1988}.
\begin{Proposition}\label{prop.2}
Assume that $(\theta_j)_{j\in\N}\subseteq (0,2\pi)$ satisfies:
$$
0<\lambda<\liminfe_{j\to\infty} \frac{\theta_{j+1}}{\theta_j}\leq \limsupe_{j\to\infty} \frac{\theta_{j+1}}{\theta_j}<\mu<1,
$$
and let $\btheta:=\{\theta_j:j\in\N\}$.
There exists a (countable) family $\calR$ of standard rectangles in $\R^2$ which is totally ordered by inclusion, verifies $\inf\{\diam R:R\in\calR\}=0$ and satisfies the following property: for any $k\in\N^*$, there exists sets $\Theta_k\subseteq\R^2$ and $Y_k\subseteq\R^2$ satisfying the following conditions:
\begin{enumerate}
\item[(i)] $|Y_k|\geq \kappa(\mu)\cdot k\lambda^{-k} |\Theta_k|$;
\item[(ii)] for any $x\in Y_k$, one has $M_{r_{\btheta}\calR} \chi_{\Theta_k} f(x)\geq \kappa'(\mu)\lambda^k$;
\end{enumerate}
here,  $\kappa(\mu)>0$ and $\kappa'(\mu)>0$ are two constants depending only on $\mu$.
\end{Proposition}
\begin{proof}
Define $\calR=\{Q_k:k\in\N^*\}$ where the sequence $(Q_k)_{k\in\N^*}$ is defined inductively as follows.
We choose $Q_1=[0,L_1]\times [0,\ell_1]$ and $\btheta_1\subseteq\btheta$ associated to $k=1$ and $\epsilon=1$ according to Lemma~\ref{cl.2}. Assuming that $Q_1,\dots, Q_k$ have been constructed, for some integer $k\in\N^*$, we choose $Q_{k+1}=[0,L_{k+1}]\times [0,\ell_{k+1}]$ and $\btheta_{k+1}$ associated to $k+1$ and $\epsilon=\min(\ell_{k},1/k)$ according to Lemma~\ref{cl.2}. Since the sequence $(Q_k)_{k\in\N^*}$ is a nonincreasing sequence of rectangles, it is clear that $\calR$ is totally ordered by inclusion. It is also clear by construction that one has $\inf\{\diam R:R\in\calR\}=0$.

Now fix $k\in\N^*$ and define $\Theta_k:=B(0,\ell_k)$ and $Y_k:=\bigcup_{\theta\in\btheta_k} r_\theta Q_k$. Compute hence, using [Claim~\ref{cl.2}, (ii) and (iii)]:
$$
|Y_k|\geq \frac 12 k L_k\ell_k=\frac{1}{2\pi} k \frac{L_k}{\ell_k}\cdot \pi\ell_k^2 \geq \frac{c(\mu)}{2\pi}\cdot k\lambda^{-k} |\Theta_k|,
$$
so that (i) is proved in case one lets $\kappa(\mu):=\frac{c(\mu)}{2\pi}$.

\begin{figure}[t]\begin{center}\includegraphics[width=9cm]{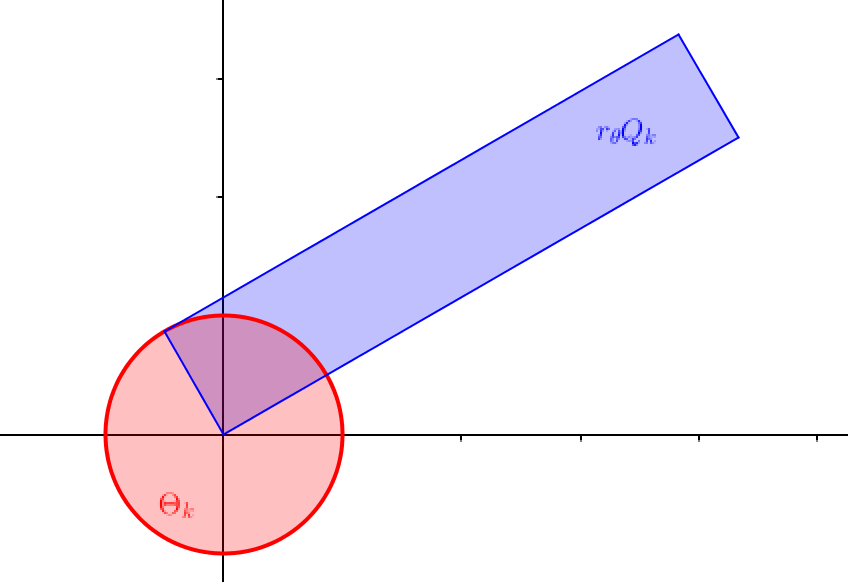} \caption{The intersection $\Theta_k\cap r_\theta Q_k$}\label{fig.2}\end{center}\end{figure}
For $x\in Y_k$, choose $\theta\in\btheta_k$ for which one has $x\in r_\theta Q_k$ and observe that one has (see Figure~\ref{fig.2}): 
$$
M_{r_{\btheta}\calR} \chi_{\Theta_k}(x)\geq\frac{|\Theta_k\cap r_\theta Q_k|}{|Q_k|}=\frac{\frac 14 \cdot \pi\ell_k^2}{L_k\ell_k}=\frac{\pi}{4} \cdot\frac{\ell_k}{L_k}\geq \frac{\pi}{4d(\mu)} \lambda^k,
$$
which finishes the proof of (ii) if we let $\kappa'(\mu):=\frac{\pi}{4d(\mu)}$.
\end{proof}

For our purposes, an \emph{Orlicz function} is a convex and increasing function $\Phi:[0,\infty)\to [0,\infty)$ satisfying $\Phi(0)=0$; we then let $L^\Phi(\R^2)$ denote the set of all measurable functions $f$ in $\R^2$ for which $\Phi(|f|)$ is integrable (for $\Phi(t)=t^p$, $p\geq 1$ this yields the usual Lebesgue space $L^p(\R^2)$, while for $\Phi(t)=\Phi_0(t):=t(1+\log_+t)$ we get the Orlicz space $L\log_+L(\R^2):=L^{\Phi_0}(\R^2)$). Recall that a sublinear operator $T$ is said to be \emph{of weak type $(\Phi,\Phi)$} in case there exists a constant $C>0$ such that, for all $f\in L^\Phi(\R^2)$ and all $\alpha>0$, one has:
$$
|\{x\in\R^2:Tf(x)>\alpha\}|\leq\int_{\R^2} \Phi\left(\frac{|f|}{\alpha}\right).
$$
Whenever $\Phi(t)=t^p$ for $p\geq 1$, we shall say that $T$ has \emph{weak type $(p,p)$}.\\

The next result specifies the announced Theorem~\ref{thm.1}. It is mainly a consequence of the preceding proposition and some standard techniques as developed in {\scshape Moonens} and {\scshape Rosenblatt} \cite{MR2012}.
\begin{Theorem}\label{thm.2}
Assume that $(\theta_j)_{j\in\N}\subseteq (0,2\pi)$ satisfies:
$$
0<\liminfe_{j\to\infty} \frac{\theta_{j+1}}{\theta_j}\leq\limsupe_{j\to\infty}\frac{\theta_{j+1}}{\theta_j}<1,
$$
and let $\btheta:=\{\theta_j:j\in\N\}$.
There exists a (countable) family $\calR$ of standard rectangles in $\R^2$ with $\inf\{\diam R:R\in\calR\}=0$, satisfying the following conditions:
\begin{enumerate}
\item[(i)] $M_\calR$ has weak type $(1,1)$, and hence the associated differentiation basis $\calB$ differentiates $L^1(\R^2)$;
\item[(ii)] for any Orlicz function $\Phi$ satisfying $\Phi=o(\Phi_0)$ at $\infty$, $M_{r_{\btheta}\calR}$ fails to be of weak type $(\Phi,\Phi)$. In particular, $M_{r_{\btheta} \calR}$ fails to have weak type $(1,1)$, and hence the associated differentiation basis $\calB_{\btheta}$ fails to differentiate $L^1(\R^2)$.
\end{enumerate}
\end{Theorem}
\begin{proof}
Begin by choosing real numbers $0<\lambda<\mu<1$ such that one has:
$$
0<\lambda<\liminfe_{j\to\infty} \frac{\theta_{j+1}}{\theta_j}\leq \limsupe_{j\to\infty} \frac{\theta_{j+1}}{\theta_j}<\mu<1,
$$
and keep the notations of Proposition~\ref{prop.2}.

Let now $\calR$ be the family of rectangles given by Proposition~\ref{prop.2}. Observe first that, since $\calR$ is totally ordered by inclusion, it follows \emph{e.g.} from \cite[Claim~1]{STOKOLOS2005} that $M_\calR$ satisfies a weak $(1,1)$ inequality.

In order to show (ii), define, for $k$ sufficiently large, $f_k:=[1/\kappa'(\mu)]\cdot \lambda^{-k} \chi_{\Theta_k}$, where $\Theta_k$ and $Y_k$ are associated to $k$ and $\calR$ according to Proposition~\ref{prop.2}.
\begin{Claim}\label{cl.Mphi} For each sufficiently large $k$, we have:
$$
|\{x\in\R^2:M_\calR f_k(x)\geq 1\}|\geq c_1(\lambda,\mu)\int_{\R^2} \Phi_{0}(f_k),
$$
where $c_1(\lambda,\mu):=\frac{2\log\frac{1}{\lambda}}{\kappa(\mu)\cdot \kappa'(\mu)}$ is a constant depending only on $\lambda$ and $\mu$.
\end{Claim}
\begin{proof}[Proof of the claim]
To prove this claim, one observes that for $x\in Y_k$ we have $M_\calR f_k(x)\geq 1$ according to [Proposition~\ref{prop.2}, (ii)]. Yet, on the other hand, one computes, for $k$ sufficiently large:
\begin{multline*}
\int_{\R^2} \Phi_{0}(f_k)\leq \frac{1}{\kappa'(\mu)}\cdot \lambda^{-k} |\Theta_k| \left[1-\log_+\kappa'(\mu)+k\log\frac{1}{\lambda}\right]\\ \leq \frac{2\log\frac{1}{\lambda}}{\kappa'(\mu)}\cdot k\lambda^{-k} |\Theta_k|\leq
c_1(\lambda,\mu)  \cdot |Y_k|,\end{multline*}
and the claim follows.
\end{proof}

\begin{Claim}
For any $\Phi$ satisfying $\Phi=o(\Phi_{0})$ at $\infty$ and for each $C>0$, we have:
$$
\lim_{k\to\infty}\frac{\int_{\R^2} \Phi_{0}(|f_k|)}{\int_{\R^2} \Phi(C|f_k|)}=\infty.
$$
\end{Claim}
\begin{proof}[Proof of the claim]
Compute for any $k$:
\begin{eqnarray*}
\frac{\int_{\R^2} \Phi(C|f_k|)}{\int_{\R^2} \Phi_{0}(|f_k|)} & = & \frac{\Phi(\lambda^{-k}C/\kappa'(\mu))}{\Phi_{0}(\lambda^{-k}/\kappa'(\mu))}\\
& = & \frac{\Phi( \lambda^{-k}C/\kappa'(\mu))}{\Phi_{0}(\lambda^{-k}C/\kappa'(\mu))}
 \frac{\Phi_{0}( \lambda^{-k}C/\kappa'(\mu))}{\Phi_{0}(\lambda^{-k}/\kappa'(\mu))},
\end{eqnarray*}
observe that the quotient $ \frac{\Phi_{0}( \lambda^{-k}C/\kappa'(\mu))}{\Phi_{0}(\lambda^{-k}/\kappa'(\mu))}$ is 
bounded as $k\to\infty$ by a constant independent of $k$, while by assumption the quotient 
$\frac{\Phi( \lambda^{-k}C/\kappa'(\mu))}{\Phi_{0}(\lambda^{-k}C/\kappa'(\mu))}$ tends to zero as $k\to\infty$. The claim is proved.
\end{proof}

We now finish the proof of Theorem~\ref{thm.2}. To this purpose, fix $\Phi$ an Orlicz function satisfying $\Phi=o(\Phi_{0})$ at $\infty$ 
and assume that there exists a constant $C>0$ such that, for any $\alpha>0$, one has:
$$
|\{x\in\R^2:M_\calR f(x)>\alpha\}|\leq\int_{\R^2} \Phi\left(\frac{C|f|}{\alpha}\right).
$$
Using Claim~\ref{cl.Mphi}, we would then get, for each $k$ sufficiently large:
$$
0<c_1(\lambda,\mu)\int_{\R^2}\Phi_{0}(f_k)\leq \left|\left\{x\in\R^2:M_\calR f_k(x)> \frac 12 \right\}\right|\leq \int_{\R^n} \Phi({2Cf_k}),
$$ contradicting the previous claim and proving the theorem.\end{proof}
\begin{Remark}
If we are solely interested in the weak $(1,1)$ behaviour of the maximal operators $M_\calR$ and $M_{r_{\btheta}\calR}$, observe that Theorem~\ref{thm.2} in particular applies to $\Phi(t)=t$, ensuring that the maximal operator $M_{r_{\btheta}\calR}$ also fails to have weak type $(1,1)$.

Moreover, as pointed out by the referee, the construction, given a sequence of distinct angles $\btheta=(\theta_j)_j\subseteq (0,\pi/2)$, of a countable family $\calR$ of rectangles for which $M_\calR$ is of weak type $(1,1)$ while $M_{r_{\btheta}\calR}$ is not, can be done almost immediately from Lemma~\ref{cl.1}~---~and does not require a growth condition on the sequence $\btheta$.

To see this, observe that for each $k$, it is easy, according to Lemma~\ref{cl.1} and making $L_k/\ell_k \gg 1$ large enough, to construct a rectangle $Q_k=[0,L_k]\times [0,\ell_k]$ such that the rectangles $r_{\theta_j}Q_{k,+}$, $0\leq j\leq k$ are pairwise disjoint. We can also inductively construct $(Q_k)$ such that one has $Q_{k+1}\subseteq Q_k$ for all $k\in\N$. Hence $\calR:=\{Q_k:k\in\N\}$ is totally ordered by inclusion, ensuring that $M_\calR$ has weak type $(1,1)$.

On the other hand, define for $k\in\N$ a function $f_k:=|Q_k|\frac{\chi_{B(0,\ell_k)}}{|B(0,\ell_k)|}$.
For all $x\in Y_k:=\bigcup_{j=0}^k r_{\theta_j} Q_k$, choose an integer $0\leq j\leq k$ for which one has $x\in r_{\theta_j} Q_k$ and compute (see Figure~\ref{fig.2} again):
$$
M_{r_{\btheta}\calR} f_k(x)\geq \frac{|Q_k|}{|B(0,\ell_k)|} \frac{|B(0,\ell_k)\cap r_{\theta_j}Q_k|}{|Q_k|}=\frac 14.
$$
It hence follows that one has:
\begin{multline*}
(k+1)\|f_k\|_1=(k+1) |Q_k|=2 (k+1) |Q_{k,+}|\\ \leq 2 |Y_k|\leq 2\left|\left\{x\in\R^2:M_{r_{\btheta}\calR}f_k(x)\geq\frac 14\right\}\right|,
\end{multline*}
so that $M_{r_{\btheta}\calR}$ cannot have weak type $(1,1)$.
\end{Remark}
\begin{Remark}
In [Theorem~\ref{thm.2}, (ii)], it is not clear to us whether or not the space $L\log L(\R^2)$ is sharp; we don't know, for example, whether or not $\calB_{\btheta}$ differentiates $L\log^{1+\epsilon} L(\R^2)$ for $\epsilon\geq 0$.
\end{Remark}

\paragraph{Acknowledgements.} I would like to thank my colleague and friend Emma D'Aniello for her careful reading of the first manuscript of this paper. I also express my gratitude to the referee for his/her careful reading of the paper and his/her nice suggestions which were of a great help to improve it.

\end{document}